\documentclass[10pt, english]{amsart}

\usepackage{amsmath,amssymb,enumerate}

\usepackage[T1]{fontenc}
\usepackage[all]{xy}

\usepackage{babel}
\usepackage{amstext}
\usepackage{amsmath}
\usepackage{amsfonts}
\usepackage{latexsym}
\usepackage{ifthen}

\usepackage{xypic}
\xyoption{all}
\newtheorem{lemma1}{}

\newenvironment{theorem}{\begin{lemma1}{\bf Theorem.}}{\end{lemma1}}

\newenvironment{proposition}{\begin{lemma1}{\bf Proposition.}}{\end{lemma1}}

\newenvironment{conjecture}{\begin {lemma1}{\bf Conjecture.}}{\end{lemma1}}

\newenvironment{the local obstruction - setup}{\begin{lemma1}{\bf The local obstruction - setup.}}{\end{lemma1}}

\newenvironment{remark*}{{\bf Remark.}}{}
\newenvironment{remarks*}{{\bf Remarks.}}{}
\newenvironment{example*}{{\bf Example.}}{}
\newenvironment{assumption*}{{\bf Assumption.}}{}

\newcommand{\Q}{\ensuremath{\mathbb{Q}}}

\usepackage{eurosym}

\makeatletter
\ifnum\@ptsize=0  \fi \ifnum\@ptsize=2  \fi 

\newcommand\sF{{\mathcal F}}
\newcommand\sG{{\mathcal G}}

\newcommand\sL{{\mathcal L}}

\newcommand\sO{{\mathcal O}}

\usepackage{xcolor}

\title{Erratum and addendum to the paper: Abundance for K\"ahler threefolds} 
\date{April 20, 2023}

\author{Fr\'ed\'eric Campana}
\author{Andreas H\"oring}
\author{Thomas Peternell}

\address{Fr\'ed\'eric Campana, Institut Elie Cartan,
Universit\'e Henri Poincar\'e,
BP 239,
F-54506. Vandoeuvre-les-Nancy C\'edex,
et: Institut Universitaire de France\\
}
\email{frederic.campana@univ-lorraine.fr}

\address{Andreas H\"oring, Laboratoire de Math{\'e}matiques J.A. Dieudonn{\'e},
UMR 7351 CNRS, Universit{\'e} de Nice Sophia-Antipolis, 06108 Nice Cedex 02, France        
}
\email{hoering@unice.fr}

\address{Thomas Peternell, Mathematisches Institut, Universit\"at Bayreuth, 95440 Bayreuth, 
Germany}

\email{thomas.peternell@uni-bayreuth.de}

\begin{document}

\begin{abstract} 
In this text we signal a serious gap in the proof of the main theorem of our paper and explain which parts of the statement remain valid.
In fact, the main theorem remains valid unless possibly the variety  does not admit positive-dimensional subvarieties through a very general point and 
is not bimeromorphic to a quotient of a torus. This latter case would be ruled out by a Chern class inequality which holds in the algebraic case but is still unknown 
in the K\"ahler setting. 
\end{abstract}

\maketitle


The proof of Theorem 8.2 in \cite{a26} contains a gap. We explain here how to fix this in most cases, the only remaining case being "simple" threefolds $X$ of Kodaira dimension $0$,  hence through a very general point of $X$ there is no proper positive-dimensional compact analytic subvariety. 
In this case it is expected that $X$ is bimeromorphic to a torus quotient. We explain further how to reduce Theorem 8.2 to an orbifold Chern class inequality of Bogomolov-Gieseker type which in the algebraic setting is well-known and proved by reducing to dimension $2$ by taking hyperplane sections.  For recent developments towards the inequality in the K\"ahler case, we refer to \cite{CGNPPW23}.

The issue is concerned with the proof of Equality (33).  To be precise, we are in the following situation. 
Let $X$ be a normal $\mathbb Q$-factorial compact K\"ahler threefold such that $(X,0)$ is klt. Further, $X$ carries a
divisior $D \in \vert mK_X \vert$ with the following properties.
\begin{enumerate}
\item Set $B := {\rm Supp} D$. Then $(X,B) $ is lc and $X \setminus B$ has terminal singularities.
\item $K_X + B$ is nef with $\nu(K_X + B) = 2$; further, $\kappa (X) = \kappa (K_X+B)$.
\item For all irreducible components $T \subset B$, we have $(K_X + B)_T \ne 0$.
\item $(K_X + B) \cdot K_X^2 \geq 0$.
\end{enumerate}
Then we claim that  $\kappa (X) \geq 1$. 

To prove this, we consider a terminal modification $\mu: X' \to X$, set $L : = m(K_X+B) $; $L' = \mu^*(L) $ 
 and things come down to prove Equality (33) in \cite{a26}, stating 
\begin{equation} 
\label{Eq1} L' \cdot (K_{X'}^2 + c_2(X')) \geq 0. 
\end{equation} 

To prove Equation (\ref{Eq1}), we claimed that we may assume that $X$ has canonical singularities in codimension two. This however is not true in general. 
Once Equation (\ref{Eq1}) is settled (which holds  when 
$X$ has canonical singularities), the proof given in \cite{a26} is complete.
Thus we have to treat the case that $X$ has (possibly) non-canonical singularities. In this case $X$ has quotient singularities in codimension two. 
Instead of taking a terminalization, we let $\mu: X' \to X$ be a desingularization. 
We proceed following the lines of \cite[chapter14]{Uta92}. As in \cite[Lemma 14.3.1]{Uta92}, things come down to prove 
\begin{equation} \label{Eq2} L \cdot (K_X^2 + \hat c_2(X)) \geq 0. \end{equation} 
Since
$$ L \cdot (K_X+B) \cdot B = 0,$$ this comes 
by \cite[Lemma 14.3.2]{Uta92} down to show
\begin{equation} \label{Eq3} L \cdot \hat c_2(\Omega_X(\log B)) \geq 0.\end{equation} 

To prove Equation \ref{Eq3}, we proceed as in \cite{a26} in an orbifold context. 
We refer to \cite{GK20}  for the notion of orbifold Chern classes in an analytic context. 
In particular, we 
may define the intersection $\omega \cdot \hat c_2(\sF)$ for a K\"ahler form $\omega $ on $X$ and a reflexive sheaf $\sF$ on $X$ which is a
Q-sheaf on $X \setminus S$:
the orbifold Chern class $\hat c_2(\sF) \in H^4(X \setminus S,\mathbb R)$ extends uniquely to a class $\gamma $ on $X$, and we simply define
$$ \omega \cdot \hat c_2(\sF) := [\omega] \cdot \gamma,$$
where $[\omega] \in H^2(X,\mathbb R)$ is the class of $\omega$.

In a first step, we need the following orbifold version of \cite[Theorem 7.1]{a26}.

\begin{conjecture} \label{conj1}
Let $X$ be a normal compact K\"ahler threefold with only isolated non-quotient singularities,
and let $\alpha$ be a K\"ahler class on $X$.  Let $\sF$ be an $\alpha-$stable non-zero torsion-free coherent $Q$-sheaf on $X$.
Then we have
$$
\alpha \cdot \hat c_2(\sF)  \geq \big(\frac{r-1}{2r}\big)  \alpha \cdot  c_1^2(\sF).
$$ 
\end{conjecture}

Of course, Conjecture \ref{conj1} should be true in all dimensions under the assumption that $X$ has only quotient singularities in codimension $2$. 
If $X$ is projective, see \cite[chap.10]{Uta92} for the surface case and \cite[Thm.6.1]{GKP19c}

If Conjecture \ref{conj1} holds for klt spaces,  it yields  the following orbifold version of
\cite[Theorem 7.2]{a26}.

\begin{theorem} \label{thm2} 
Let $(X, \omega)$ be a compact K\"ahler threefold with isolated non-quotient singularities, and let $\sF$ be a non-zero reflexive coherent $Q$-sheaf
on $X$ such that $\det \sF$ is $\Q$-Cartier. Suppose that there exists a pseudoeffective class $P \in N^1(X)$ such that
$$
L := c_1(\sF) + P
$$
is a nef class. Suppose furthermore that for all $0 < \varepsilon \ll 1$ the sheaf $\sF$ is $(L+\varepsilon \omega)$-generically nef. Then we have
$$
L \cdot \hat c_2(\sF) \geq \frac{1}{2} (L \cdot c_1^2(\sF) - L^3).
$$
In particular, if $L \cdot c_1^2(\sF) \geq 0$ and $L^3=0$, then 
\begin{equation} \label{nonnegative}
L \cdot \hat c_2(\sF) \geq 0.
\end{equation}
\end{theorem}

We aim to apply Theorem \ref{thm2} for the reflexive sheaf $\sF = \Omega_X(\log B)$ to prove Equation \ref{Eq3}. Therefore, we need to check generic nefness of $\sF$ which 
follows from the generic nefness of $\Omega_X$. 
This is done by the following adaption of \cite[Proposition 8.2]{a26}.

\begin{proposition} \label{prop1}
Let $X$ be a normal  non-algebraic compact K\"ahler space of dimension $n$ with klt singularities. 
Suppose that $\kappa(X) \geq 0$.
Then $\Omega_X$ is generically nef with respect to any nef class $\alpha$, i.e. for every torsion-free quotient sheaf
$$
\Omega_X \rightarrow \mathcal Q \rightarrow 0,
$$
we have $\alpha^{n-1} \cdot c_1(\mathcal Q) \geq 0$.
\end{proposition} 

\begin{proof} 
We proceed as in \cite{a26}. Let $\pi: \hat X \to X$ be a desingularization. Using the same notations as in the proof of  \cite[Proposition 8.2]{a26},
we need to show that 
$$ \pi^*(\alpha) ^{n-1} \cdot c_1(\hat Q) \geq 0.$$ To do this, we apply Enoki's theorem \cite[Theorem 1.4]{En88}, and need to write as $\mathbb Q$-divisors, 
$$ K_{\hat X} = L + D$$
with $L$ nef and $D$ effective. 
Setting $L = 0$, we need to show that $\kappa (\hat X) \geq 0$. Since $X$ is klt, this is not automatic from $\kappa (X) = \kappa (K_X) \geq 0$. 
Assuming to the contrary that $\kappa (\hat X) = - \infty$. Then by \cite{a21}, $\hat X$ is uniruled. Since $\hat X$ is not algebraic, $\hat X$ cannot be rationally connected, hence we have 
a rational almost holomorphic (MRC) quotient 
$$ \hat f: \hat X \dasharrow \hat S $$
with $\dim \hat S = 1,2$. Since $H^2(\hat X,\sO_{\hat X}) \ne 0$, we must have $\dim \hat S = 2$. Since $\hat X$ is not algebraic, the exceptional set of $\pi$ cannot surject onto $\hat S$ and 
therefore $\hat f$ descends to an almost holomorphic map $f: X \dasharrow \hat S$ whose general fiber is a smooth rational curve. This contradicts $\kappa (K_X) \geq 0$. 

\end{proof}

\vskip .5cm 

\begin{proof}[Proof of Theorem 8.2 in the non-simple case.]

In the rest of the paper we point out how directly prove Theorem 8.2 in \cite{a26} - without using a Bogomolov-Gieseker inequality -  in case $X$ has algebraic dimension $a(X) \geq 1 $ or carries a family of curves covering $X$.  For the theory of algebraic reductions, which we will use freely, we refer e.g. to \cite{CP94}. 
We will write $\sL = \sO_X(L)$. 

{\it 1st case:}  We assume that there exists an almost holomorphic elliptic fibration $$f: X \dasharrow S,$$ over a smooth 
K\"ahler surface. Since $f$ cannot be a Moishezon morphism,  $\sL$ must be numerically trivial on the general fiber of $f$. 
We now consider the graph of the closure of the family of elliptic curves (fibers of $f$) in $X$ and obtain a birational map $p: Z \to X$ and an equidimensional morphism 
$q: Z \to T$, where $T$ is the parameter space of the family (birational to $S$). We may assume $T$ smooth and $Z$ normal. The line bundle $\tilde \sL = p^*(\sL)$ 
is $q$-numerically trivial, actually trivial on the general fiber of $q$ since $H^0(X;\sL) \ne 0$. Using Zariski's lemma \cite[III, Lemma 8.2]{BHPV04} we obtain that (up to replacing $L$ by a multiple) we have $\sL \simeq q^*(\sG)$ with a line bundle $\sG$ on $T$.
Since $\nu(\tilde \sL) = 2$, $\sG$ is big and nef. Thus $S$ is Moishezon and $\kappa (\tilde \sL) = 2$, hence $\kappa (\sL) = 2$. 

Note that this first case settles in particular the case where the algebraic dimension
$a(X)$ is equal to two: in this case the algebraic reduction defines an elliptic fibration as above.

\vskip .2cm \noindent 
{\it 2nd case: assume that $a(X) = 1$.}  Here the algebraic reduction is a holomorphic fiber space $$f: X \to C$$ over a smooth curve $C$. Let $X_c$ be a general fiber of $f$, a klt surface, and $\hat X_c \to X_c$ the minimal desingularization. 
Since $X$ is not uniruled and since $\kappa (\hat X_c) \leq 0$ for the general fiber of any algebraic reduction, we conclude $\kappa (\hat X_c) = 0$, hence $\kappa (X_c) \geq 0$. 
If the general fiber $X_c$ is algebraic, $f$ cannot have a generic multisection, hence $X_c$ must be smooth and minimal. Thus $B_{\vert X_c} = 0$ and $\sL \vert_{X_c} \simeq \sO_{X_c}$.
Since we may assume $ h^0(X, \sL^{\otimes k}) = 1
$ for all $k \geq 1$, we have $f_*(\sL) = \sO_C$ (anyway we must have $C = \mathbb P_1$ by $C_{3,1}$),
 and thus $\sL = \sO_X(\sum a_j D_j)$ with $a_j > 0$ and $D_j$ fiber components. Since $\sL$ is nef and numerically trivial on the general fibre, this implies $\nu(\sL) = 1$.
This contradicts our assumption $\nu(\sL) = 2$. 

If the general fiber $X_c$ is not algebraic, we have $L^2 \cdot X_c = 0$. In fact,
$L_{\vert X_c}$ cannot be big, otherwise $X$ would be projective. 
Then we see immediately that 
$$ K_{X_c}^2 = K_{X_c} \cdot B_{\vert X_c} = B_{\vert X_c}^2 = 0.$$
Further, $K_{X_c}$ must be nef, otherwise $X_c$ has a minimal model whose canonical divisor is big. By abundance for surfaces, $K_{X_c}$ is semiample.

If the Kodaira fibration defines an elliptic fibration on $X_c$ we obtain a covering family of elliptic curves on $X$ and conclude as in the previous case. 

If $K_{X_c}$ is torsion, the same holds for $\sL_{\vert X_c}$. 
Hence $f_*(\sL^k) = \sO_C$ for some $k$, and we conclude as before. 

\vskip .2cm \noindent 
{\it 3rd case: assume that $a(X) = 0$ and $X$ is not simple}. In this case  $X$ is Kummer, i.e., bimeromorphic to a quotient of a torus - and there is nothing to prove - or $X$ has a covering family of elliptic curves, see \cite{CP00}. 
We consider again - as in the first case  - the graph $p: Z \to X$. Since $a(X) = 0$, $X$ has only finitely many divisiors and therefore $p$ is generically finite. 
As before, we deduce  $\kappa(p^*(\sL)) = 2$, hence $\kappa (\sL) = 2$. 

\end{proof}


\providecommand{\bysame}{\leavevmode\hbox to3em{\hrulefill}\thinspace}
\providecommand{\MR}{\relax\ifhmode\unskip\space\fi MR }
\providecommand{\MRhref}[2]{%
  \href{http://www.ams.org/mathscinet-getitem?mr=#1}{#2}
}
\providecommand{\href}[2]{#2}

\end{document}